\documentclass[12pt]{article}
	\usepackage{amsmath,fullpage,amsthm,amssymb,xcolor,graphicx}
	\newtheorem{prop}{Proposition}[section]
	\newtheorem{theorem}{Theorem}[section]
	\newtheorem{corollary}{Corollary}[section]
	\newtheorem{lemma}{Lemma}[section]
	\newtheorem{definition}{Definition}[section]
	
	\newtheorem{example}{Example}[section]

	\DeclareMathOperator{\spec}{spec}
		\DeclareMathOperator{\Rea}{Re}
			\DeclareMathOperator{\Ima}{Im}

	\title{ Gain distance matrices for complex unit gain graphs}
\author{Aniruddha Samanta \thanks{Department of Mathematics, Indian Institute of Technology Kharagpur, Kharagpur 721302, India. Email: aniruddha.sam@gmail.com}\  \and M. Rajesh Kannan\thanks{Department of Mathematics, Indian Institute of Technology Kharagpur, Kharagpur 721302, India. Email: rajeshkannan@maths.iitkgp.ac.in, rajeshkannan1.m@gmail.com }}

\date{\today}
\begin{document}
\maketitle
\baselineskip=0.25in

\begin{abstract}
A complex unit gain graph ($ \mathbb{T} $-gain graph), $ \Phi=(G, \varphi) $ is a graph where the function $ \varphi $ assigns a unit complex number to each orientation of an edge of $ G $, and its inverse is assigned to the opposite orientation. 
In this article, we propose gain distance matrices for $ \mathbb{T} $-gain graphs. These notions generalize the corresponding known concepts of distance matrices and signed distance matrices. Shahul K. Hameed et al. introduced signed distance matrices and developed their properties. Motivated by their work, we establish several spectral properties, including some equivalences between balanced $ \mathbb{T} $-gain graphs and gain distance matrices. Furthermore, we introduce the notion of positively weighted $ \mathbb{T} $-gain graphs and study some of their properties. Using these properties, Acharya's and Stani\'c's spectral criteria for balance are deduced.  Moreover, the notions of order independence and distance compatibility are studied. Besides, we obtain some characterizations for distance compatibility.
\end{abstract}



\section{Introduction}

Let $ \Phi=(G, \varphi) $ be a connected complex unit gain graph ($ \mathbb{T} $-gain graph) on a simple  graph $ G $ with $ n $ vertices.  Let $ V(G)=\{v_1, v_2, \dots, v_n\} $ and $ E(G) $ be the vertex set and the edge set of $ G $, respectively. If two vertices $ v_i $ and $ v_j $ are connected by an edge, then we write $ v_i\sim v_j $. If $ v_i\sim v_j $, then the  edge between them is denoted by $ e_{i,j} $.  The \emph{adjacency matrix } $ A(G) $ of a graph $ G $ is a symmetric matrix whose $ (i,j)th $ entry is $ 1 $ if $ v_i\sim v_j $ and zero otherwise. A \emph{path} $ P $ in $ G $ between the vertices $ s$ and $t $ is denoted by $ sPt $.
The \emph{distance} between  two vertices $ s$ and $t$ in $ G $ is the length of the shortest path between $ s$ and $t $, and is denoted by $ d_{G}(s,t) $ (or simply $ d(s,t) $). The \emph{distance matrix} of an undirected graph $ G $, denoted by $ D(G) $, is the symmetric $ n\times n$ matrix whose $ (i,j) $th entry is $ d(v_i,v_j) $. The distance matrix of an undirected graph has been widely studied in the literature, see \cite{Dist_Bapat,Graham, Lovasz, Graham_1971} and the references therein.

Let $ G $ be a simple undirected graph. An oriented edge from the vertex $ v_s $ to the vertex $ v_t $ is denoted by $ \overrightarrow{e}_{s,t} $. For each undirected edge $ e_{s,t}\in E(G) $, there is a pair of oriented edges $ \overrightarrow{e}_{s,t} $ and $ \overrightarrow{e}_{t,s} $. The collection $ \overrightarrow{E}(G):=\{ \overrightarrow{e}_{s,t},\overrightarrow{e}_{t,s}: e_{s,t}\in E(G)\} $ is  the \emph{oriented edge set associated with $ G $}. Let $ \mathbb{T}=\{ z\in \mathbb{C}: |z|=1\} $.  A \emph{complex unit gain graph (or $ \mathbb{T} $-gain graph)} on a simple graph $ G $ is an ordered pair $ (G, \varphi) $, where the gain function $ \varphi: \overrightarrow{E}(G) \rightarrow \mathbb{T} $ is a mapping  such that $ \varphi( \overrightarrow{e}_{s,t}) =\varphi(\overrightarrow{e}_{t,s})^{-1}$, for every $ e_{s,t}\in E(G) $. A $ \mathbb{T} $-gain graph $ (G, \varphi) $ is  denoted by $ \Phi $. The \emph{ adjacency matrix} of  a $ \mathbb{T} $-gain graph $ \Phi=(G, \varphi)$ is  a Hermitian  matrix, denoted by $ A(\Phi)$ and its $ (s,t)th $ entry is defined as follows:
$$a_{st}=\begin{cases}
	\varphi(\overrightarrow{e}_{s,t})&\text{if } \mbox{$v_s\sim v_t$},\\
	0&\text{otherwise.}\end{cases}$$
 The spectrum and the spectral radius of $ \Phi $ are the spectrum and the spectral radius of $ A(\Phi) $  and denoted by $ \spec(\Phi) $ and $ \rho(\Phi) $, respectively.
 A \emph{signed graph}  is a graph  $ G $ together with a signature function $ \psi : E(G) \rightarrow \{\pm 1\} $, and is denoted by $\Psi= (G, \psi) $. The adjacency matrix of $ \Psi $, denoted by $ A(\Psi) $, is an $ n\times n$ matrix whose $ (i,j) $th entry is $ \psi(e_{i,j}) $. Therefore, a signed graph can be considered as a $ \mathbb{T} $-gain graph $ \Psi=(G, \psi) $, where $ \psi $ is a signature function. The notion of adjacency matrix of $\mathbb{T}$-gain graphs generalize the notion of adjacency matrix of undirected graphs, adjacency matrix of  signed graphs and the Hermitian adjacency matrix of a mixed graph.  For more information about the properties of gain graphs and $\mathbb{T}$-gain graphs, we refer to \cite{  reff1, Our-paper-2, gain-genesis2, Zaslav}.

Let $ \Psi=(G, \psi ) $ be a signed graph. The sign of a path in $ \Psi $ is the product of sign of all edges of the path \cite{Zas4}. Recently, in \cite{Sign_distance} the authors introduced the notion of signed distance matrices $ D^{\max}(\Psi) $ and $ D^{\min}(\Psi) $ for a signed graph $ \Psi $.
\begin{definition}[{\cite[Definition 1.1]{Sign_distance}}]{\rm
Let	$ \Psi=(G, \psi ) $ be a signed graph with vertex set $ V(G) =\{ v_1, v_2, \dots, v_n\}$. Then two auxiliary signs are defined as follows:
\begin{itemize}
	\item[(a)] $ \psi_{\max}(v_i,v_j)=-1 $ if all shortest $ v_iv_j $-paths are negative, $ +1 $ otherwise,
	\item[(b)] $ \psi_{\min}(v_i,v_j)=+1 $ if all shortest $ v_iv_j $-paths are positive, $ -1 $ otherwise.
\end{itemize}
The two signed distance matrices are defined as follows:
\begin{itemize}
	\item[(a)] $ D^{\max}(\Psi)=(d_{\max}(v_i,v_j))_{n\times n}$,
	\item[(b)] $ D^{\min}(\Psi)=(d_{\min}(v_i,v_j))_{n\times n} $,
\end{itemize}
where $ d_{\max}(v_i,v_j)=\psi_{\max}(v_i,v_j)d(v_i,v_j)$ and $ d_{\min}(v_i,v_j)=\psi_{\min}(v_i,v_j)d(v_i,v_j)$.}
\end{definition}

A signed graph $ \Psi $ is distance compatible if and only if $ D^{\max}(\Psi)=D^{\min}(\Psi) $. A characterization of balanced signed graph in terms of signed distance matrices is obtained in \cite{Sign_distance}. For more about signed distance matrices, see \cite{Sign_distance, shijin2020signed}.

In this article, we introduce the notion of gain distance matrices $ D^{\max}_{<} (\Phi)$ and $ D^{\min}_{<} (\Phi) $ for a $ \mathbb{T} $-gain graph $ \Phi=(G, \varphi) $ associated with an ordered vertex set $ (V(G), <) $. These concepts generalize the notions of signed distance matrices of signed graphs and distance matrices of undirected graphs.  We define positively weighted $ \mathbb{T} $-gain graphs and establish two new characterizations for balance of gain graphs. Acharya's Spectral criterion and Stani\'c's spectral criterion are particular cases of these characterizations. Besides, we introduce two properties of a $ \mathbb{T} $-gain graph, ordered-independence, and distance compatibility to gain distance matrices. Thereupon we establish two characterizations for the balance of $ \mathbb{T} $-gain graphs in terms of gain distance matrices and distance compatibility properties. Subsequently, we present some results on the characterization of distance compatibility.

This paper is organized as follows: In section \ref{prelim}, we collect needed  known definitions and results. In section \ref{gain_distance}, we define the notion of gain distance matrices, order-independent and distance compatibility, and discuss their properties. In section \ref{Positively_weighted}, we discuss the positively weighted $ \mathbb{T} $-gain graphs and establish two spectral characterizations for the balance(Theorem \ref{th4.2}, Theorem \ref{Th3.4}). In section \ref{Char_of_balanced}, we derive two characterizations for  balance  $ \mathbb{T} $-gain graph in terms of the gain distance matrices (Theorem \ref{Th5.1}, Theorem \ref{th5.2}). In section \ref{Dist_compt}, we obtain a couple of  characterizations for distance compatible $ \mathbb{T} $-gain graphs (Theorem \ref{Th6.1}, Theorem \ref{Th6.2}, Theorem \ref{Th6.3}).
	
\section{Definitions, notation and preliminary results}\label{prelim}

Let $ G=(V(G),E(G))$ be a connected undirected graph with no loops and multiple edges, where $ V(G)=\{ v_1, v_2, \dots, v_n\} $ is the vertex set and $ E(G) $ is the edge set of $ G $. A graph $ G $ is  \emph{geodetic}, if there exists a unique shortest path between any two vertices of $ G $. Let $ \Phi=(G, \varphi) $ be a $ \mathbb{T} $-gain graph on $ G $. For $ s,t \in V(G) $, $ sPt $ denotes a path starts at $ s $ and ends at $ t $ in $ G $. In case of gain graph, $ sPt $ denotes the oriented path from the vertex $ s $ to the vertex $ t $. The gain of the path $ sPt$ is $ \varphi(sPt)=\prod\limits_{j=1}^{k}\varphi(\overrightarrow{e_j})$, where $ \overrightarrow{e_1}, \overrightarrow{e_2}, \dots, \overrightarrow{e_k} $ are the consecutive oriented edges in $ sPt$. Therefore, $\varphi(tPs)= \overline{\varphi(sPt)}$.  The gain of an oriented cycle $ \overrightarrow{C_n} $ with edges $ \overrightarrow{e_1}, \overrightarrow{e_2}, \dots, \overrightarrow{e_n} $ is $\varphi(\overrightarrow{C_n})=\prod\limits_{j=1}^{n}\varphi(\overrightarrow{e_j}) $. A cycle $ C $ is  \emph{neutral} in $ \Phi $ if $ \varphi(\overrightarrow{C})=1 $. A gain graph $ \Phi $ is  \emph{balanced}, if all cycles in $ \Phi $ are neutral. A $ \mathbb{T} $-gain graph $ \Phi $ is  \emph{anti-balanced} if $ -\Phi $ is balanced. Let $ \Rea(x) $ and $ \Ima(x) $ denote the real and imaginary part of a complex number $ x $, respectively.

A function $ \zeta: V(G) \rightarrow \mathbb{T} $ is  a \emph{switching function}. Let $ \Phi_1=(G, \varphi_1) $ and $ \Phi_2=(G, \varphi_2) $ be two $ \mathbb{T} $-gain graphs. Then $ \Phi_1 $ and $ \Phi_2 $ are  \emph{switching equivalent}, denoted by $ \Phi_1 \sim \Phi_2 $, if there exists a switching function $ \zeta $ such that $ \varphi_1(\overrightarrow{e}_{i,j})= \zeta(v_i)^{-1} \varphi_2(\overrightarrow{e}_{i,j})\zeta(v_j)$, for all $ e_{i,j}\in E(G) $. If $ \Phi_1 \sim \Phi_2 $, then $ A(\Phi_1) $ and $ A(\Phi_2) $ are diagonally similar and hence have the same spectra.

\begin{lemma}[{\cite[Corollary 3.2]{Our-paper-2}}]\label{lm2.1}
	Let $ \Phi_1 $ and $ \Phi_2 $ be two $ \mathbb{T} $-gain graphs on a connected graph $ G $ with a normal spanning tree $ T $. Then $ \Phi_1 \sim \Phi_2 $ if and only if $ \varphi_1(\overrightarrow{C_j})=\varphi_2(\overrightarrow{C_j}) $, for all fundamental cycles $ C_j $ with respect to $ T $.
\end{lemma}

 A signed graph is a $ \mathbb{T} $-gain graph $ \Psi=(G, \psi) $, where $ \psi(\overrightarrow{e}_{i,j})=1$ or $ -1 $ for $ e_{i,j}\in E(G) $.  The \emph{sign} of a path in $\Psi$ is the product of the signs (the gains) of the edges in the path.

\begin{theorem}[Harary's path criterion \cite{Harary} ]\label{lm.2.3}
	Let $ \Psi $ be a signed graph on an underlying graph $ G $. Then $ \Psi $ is balanced if and only if any pair of vertices $ s,t $, every $ st $-path have the same signature.
\end{theorem}
Let $ \mathbb{C}^{m\times n} $ denote the set of all $ m\times n $ matrices with complex entries. For $ A=(a_{ij})\in \mathbb{C}^{n \times n} $, define $ |A|=(|a_{ij}|) $. For two matrices $ A=(a_{ij})$ and $ B=(b_{ij})$, we write   $ A \leq B $ if  $ a_{ij} \leq b_{ij} $ for all $ i,j $. A matrix is \emph{non-negative}, if all entries of a matrix are non-negative.  The spectral radius of a matrix $ A $ is denoted by $ \rho(A) $.
\begin{theorem}[{\cite[Theorem 8.4.5]{horn-john2}}]\label{th2.2}
Let $ A, B \in \mathbb{C}^{n \times n}$. Suppose $ A $ is irreducible and non-negative and $ A \geq |B| $. Let $ \mu=e^{i\theta} \rho(B)$ be a given maximum modulus eigenvalue of $ B $. If $ \rho(A)=\rho(B) $, then there is a unitary diagonal matrix $ D $ such that $ B=e^{i\theta}DAD^{-1} $.
\end{theorem}

Let $ A=(a_{ij}), B=(b_{ij})\in \mathbb{C}^{m\times n} $. The \emph{Hadamard product} of $ A $ and $ B $, denoted by $ A\circ B $,  is defined as $  A\circ B =(a_{ij}b_{ij})_{m\times n}$. For any three matrices $ A, B, C $ of same order, $ (A \circ B) \circ C=A \circ (B \circ C) $.
Let us recall the following property of Hadamard product of matrices.
\begin{prop}[{\cite[Lemma 5.1.2]{Matrix_Analysis}}]\label{prop3.1}
	Let $ A,B,C $ be three $n \times n$ matrices  and $ D,E $ be two $n \times n$ diagonal matrices. Then
	\begin{equation*}
		 D(A\circ B)E=(DAE)\circ B=(DA)\circ (BE)=(AE)\circ (DB)=A\circ(DBE).
	\end{equation*}
\end{prop} 	
	
\section{Gain distance matrices}\label{gain_distance}
This section introduces the notion of gain distance matrices of $ \mathbb{T} $-gain graphs, which generalize the notion of distance matrices of undirected graphs and signed distance matrices of signed graphs.

  Let $ \Phi=(G, \varphi) $ be a connected $ \mathbb{T} $-gain graph on $ G $.  For  $ s,t\in V(G) $, $sPt$ denotes the oriented path from the  vertex $s$ to the vertex $t$.  Define three sets of paths $ \mathcal{P}(s,t), \mathcal{P}^{\max}(s,t) $ and $ \mathcal{P}^{\min}(s,t) $ as follows:

\begin{center}
	$ \mathcal{P}(s,t)=\left\{ sPt: sPt \text{ is a shortest path} \right\},$
\end{center}

\begin{center}
	$ \mathcal{P}^{\max}(s,t)=\left\{ sPt \in \mathcal{P}(s,t):   \Rea(\varphi(sPt))=\max\limits_{s\tilde{P}t \in  \mathcal{P}(s,t)}\Rea(\varphi(s\tilde{P}t)) \right\}$
\end{center}
and
\begin{center}
	$ \mathcal{P}^{\min}(s,t)=\left\{ sPt \in \mathcal{P}(s,t):   \Rea(\varphi(sPt))=\min\limits_{s\tilde{P}t \in  \mathcal{P}(s,t)}\Rea(\varphi(s\tilde{P}t)) \right\}.$
\end{center}
Note that $\mathcal{P}^{\max}(s,t) = \mathcal{P}^{\max}(t, s) $ and $\mathcal{P}^{\min}(s,t) =  \mathcal{P}^{\min}(t,s)$.

Let $ G $ be a simple graph with vertex set $  V(G) =\{ v_1, v_2, \dots, v_n\} $. We denote $ \left(V(G), <\right)$ as an ordered vertex set, where $`<`$ is a total ordering of the vertices of $ G $. An ordering $ `<_r` $ is  the \emph{reverse ordering} of $ `<` $ if $ v_i<_r v_j $ if and only if $ v_j<v_i $, for any $ i,j $. An ordering $ `<` $ is  the \emph{standard vertex ordering} if  $ v_1<v_2<\dots <v_n $.
\begin{definition}[Auxiliary gains]{\rm
		Let $ \Phi=(G, \varphi) $ be a $ \mathbb{T} $-gain graph with an ordered vertex set $( V(G), <)$. We define two auxiliary gains with respect to $ < $ as follows. 	
	\begin{enumerate}
		\item[(1)] The function $ \varphi^{<}_{\max}:V(G)\times V(G)\rightarrow \mathbb{T} $ is  the maximum auxiliary gain with respect to the vertex order $ < $  such that $ \varphi^{<}_{\max}(s,t)=\overline{\varphi^{<}_{\max}(t,s)} $ for each $ (s,t)\in V(G)\times V(G) $ and $ \varphi^{<}_{\max} $ is defined by $$ \varphi^{<}_{\max}(s,t)=	\varphi(sPt)$$
		where $ s<t $, and $sPt\in \mathcal{P}^{\max}(s,t) $ and $ \Ima(\varphi(sPt))=\max\limits_{s\tilde{P}t \in \mathcal{P}^{\max}(s,t)}\Ima(\varphi(s\tilde{P}t)).$
		\item [(2)]  The function $ \varphi^{<}_{\min}:V(G)\times V(G)\rightarrow \mathbb{T} $ is  the minimum auxiliary gain with respect to the vertex order $ < $  such that $ \varphi^{<}_{\min}(s,t)=\overline{\varphi^{<}_{\min}(t,s)} $ for each $ (s,t)\in V(G)\times V(G) $ and $ \varphi^{<}_{\min} $ is defined by $$ \varphi^{<}_{\min}(s,t)=	\varphi(sPt)$$
		where $ s<t $, and $sPt\in \mathcal{P}^{\min}(s,t) $ and $ \Ima(\varphi(sPt))=\min\limits_{s\tilde{P}t \in \mathcal{P}^{\min}(s,t)}\Ima(\varphi(s\tilde{P}t)).$
		
\end{enumerate}}
\end{definition}
Note that, for $s< t$,  $ \varphi^{<}_{\max}(s, t)  (\mbox{~resp.,~} \varphi^{<}_{\min}(s, t)  )$ is the maximum (resp., minimum) gain, with respect to the lexicographic order,  over all the shortest paths between the vertices $s$ and $t$.
\begin{definition}[Gain distances]{\rm
	Let $ \Phi=(G, \varphi) $ be a $ \mathbb{T} $-gain graph with an ordered vertex set $ (V(G), <) $. For any two vertices $ s,t\in V(G)$, there are two gain distances from the vertex $ s $ to the vertex $ t $  which are defined as follows:
	\begin{enumerate}
		\item[(1)] $ d^{<}_{\max}(s,t)=\varphi^{<}_{\max}(s,t)d(s,t),$
		\item[(2)] 	$ d^{<}_{\min}(s,t)=\varphi^{<}_{\min}(s,t)d(s,t).$
	\end{enumerate}}
\end{definition}

Next we define the gain distance matrices for $ \mathbb{T} $-gain graphs.
\begin{definition}[Gain distance matrices]{\rm
	Let $ \Phi=(G, \varphi) $ be a  $ \mathbb{T} $-gain graph with an order $  < $ on the vertex set $V(G)$, where $ V(G)=\{ v_1, v_2, \cdots, v_n\} $. The gain distance matrices $ D_{<}^{\max}(\Phi) $ and $ D_{<}^{\min}(\Phi) $ associated with $ < $ are defined as follows:
	\begin{enumerate}
		\item[(1)] $ D_{<}^{\max}(\Phi)= \left( d^{<}_{\max}(v_i, v_j)\right)$,
		\item[(2)] $ D_{<}^{\min}(\Phi)= \left( d^{<}_{\min}(v_i, v_j) \right)$.
	\end{enumerate} Here $d^{<}_{\max}(v_i, v_j)$ is the $(i, j)$th entry of $ D_{<}^{\max}(\Phi).$
}
\end{definition}

The gain distance matrices are the generalization of the distance matrix of an undirected graph and signed distance matrices of a signed graph. It is easy to see that the gain distance matrices $ D^{\max}_{<}(\Phi) $ and $ D^{\min}_{<}(\Phi) $ are Hermitian, and hence have real eigenvalues. For any pair of vertices $ (v_s,v_t) $, there are two different maximum gain distances $ d^{<}_{\max}(v_s,v_t) $ and $ d^{<}_{\max}(v_t,v_s) $ which have same absolute value but they are the complex conjugate to each other. The distance $ d^{<}_{\max}(v_s,v_t) $ is  the maximum gain distance from the vertex $ v_s $ to the vertex $ v_t $ in $ \Phi $ with respect to the vertex ordering $ (V(G), <) $. Likewise, the minimum gain distance is defined.

Now we illustrate the definitions with the following example.
\begin{figure} [!htb]
	\begin{center}
		\includegraphics[scale= 0.55]{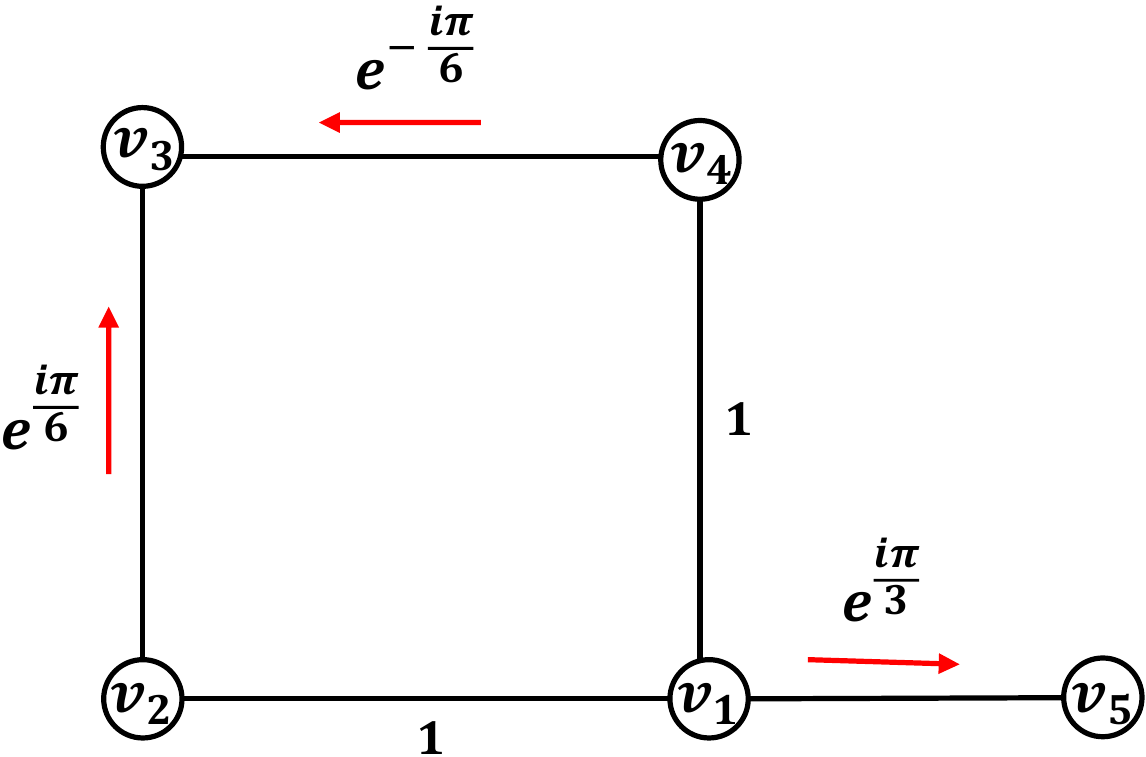}
		\caption{ $ \mathbb{T} $-gain graphs $ \Phi=(G, \varphi) $} \label{fig1.2}
	\end{center}
\end{figure}

\begin{example}\label{ex3.1}{\rm
	Let $ \Phi=(G, \varphi) $ be a $ \mathbb{T} $-gain graph shown in Figure \ref{fig1.2}. Let us consider the standard order $ < $ on vertex set $V(G)$, where $ V(G)=\{v_1,v_2, v_3, v_4, v_5\} $. Then
	\begin{equation*}
		D^{\max}_{<}(\Phi)=\left[\begin{array}{ccccc}
			0 & 1 & 2e^{\frac{i\pi}{6}} & 1 & e^{\frac{i\pi}{3}}\\
			1 & 0 & e^{\frac{i\pi}{6}} & 1 & 2e^{\frac{i\pi}{3}}\\
			2e^{-\frac{i\pi}{6}} & e^{-\frac{i\pi}{6}} & 0 & e^{\frac{i\pi}{6}} & 3e^{\frac{i\pi}{6}}\\
			1 & 1 & e^{-\frac{i\pi}{6}} & 0 & 2e^{\frac{i\pi}{3}}\\
			e^{-\frac{i\pi}{3}} & 2e^{-\frac{i\pi}{3}} & 3e^{-\frac{i\pi}{6}} & 2e^{-\frac{i\pi}{3}} & 0
		\end{array} \right].
	\end{equation*}
	Now, consider the reverse ordering $` <_r` $ of the standard  order $ `<` $. Then
	\begin{equation*}
		D^{\max}_{<_r}(\Phi)=\left[\begin{array}{ccccc}
			0 & 1 & 2e^{-\frac{i\pi}{6}} & 1 & e^{\frac{i\pi}{3}}\\
			1 & 0 & e^{\frac{i\pi}{6}} & 1 & 2e^{\frac{i\pi}{3}}\\
			2e^{\frac{i\pi}{6}} & e^{-\frac{i\pi}{6}} & 0 & e^{\frac{i\pi}{6}} & 3e^{-\frac{i\pi}{6}}\\
			1 & 1 & e^{-\frac{i\pi}{6}} & 0 & 2e^{\frac{i\pi}{3}}\\
			e^{-\frac{i\pi}{3}} & 2e^{-\frac{i\pi}{3}} & 3e^{\frac{i\pi}{6}} & 2e^{-\frac{i\pi}{3}} & 0
		\end{array} \right].
	\end{equation*}
	Here $  D^{\max}_{<}(\Phi) \ne D^{\max}_{<_r}(\Phi)$. In fact, $ \spec(D^{\max}_{<}(\Phi)) \ne \spec(D^{\max}_{<_r}(\Phi))$. Similarly,\break $ D^{\min}_{<}(\Phi) \ne D^{\min}_{<_r}(\Phi)$.}
\end{example}

\begin{definition}[Vertex order independent]{\rm
		A $ \mathbb{T} $-gain graph $ \Phi=(G, \varphi) $ is  vertex order-independent (simply, order independent), if $ D^{\max}_{<}(\Phi)=D^{\max}_{<_r}(\Phi)$ and $ D^{\min}_{<}(\Phi)=D^{\min}_{<_r}(\Phi) $, where $  < $ is the standard vertex order on $V(G)$. In this case, we define\break $D^{\max}(\Phi)=D^{\max}_{<}(\Phi)=D^{\max}_{<_r}(\Phi)$ and $ D^{\min}(\Phi)= D^{\min}_{<}(\Phi)=D^{\min}_{<_r}(\Phi)$.}
\end{definition}


Now we present a characterization for order independent $ \mathbb{T} $-gain graph.

\begin{theorem}
	Let $ \Phi=(G, \varphi)$ be a $ \mathbb{T} $-gain graph. Then $ \Phi $ is not order independent if and only if at least one of the following holds.
	\begin{enumerate}
		\item[(i)] There exists $ v_s, v_t \in V(G)$ with at least two shortest paths from $v_s$ to $v_t$ in $ \mathcal{P}^{\max}(v_s,v_t) $ have different gains.
		\item [(ii)]There exists $ v_s, v_t \in V(G)$ with at least two shortest paths  from $v_s$ to $v_t$ in $ \mathcal{P}^{\min}(v_s,v_t) $ have different gains.
	\end{enumerate}
\end{theorem}
\begin{proof}
	Let $ \Phi=(G, \varphi) $ be a $ \mathbb{T} $-gain graph with vertex set $ V(G)=\{ v_1,v_2, \cdots, v_n\} $. Let $ < $ be the standard vertex order on $ V(G) $. Suppose $ \Phi $ is not order independent, then  either  $D^{\max}_{<}(\Phi)\ne D^{\max}_{<_r}(\Phi) $ or $ D^{\min}_{<}(\Phi)\ne D^{\min}_{<_r}(\Phi)$ hold.
	Suppose  $D^{\max}_{<}(\Phi)\ne D^{\max}_{<_r}(\Phi) $. Then there exists $ v_s, v_t \in V(G) $ such that $ d_{\max}^{<}(v_s,v_t)\ne d_{\max}^{<_r}(v_s,v_t)$. Then $ \varphi_{\max}^{<}(v_s,v_t)\ne \varphi_{\max}^{<_r}(v_s,v_t)$. Let $ v_s<v_t $ and $\varphi_{\max}^{<}(v_s,v_t)=\varphi(v_sP_1v_t) =x+iy\in \mathbb{T}$, for some $ v_sP_1v_t \in \mathcal{P}^{\max}(v_s,v_t) $. It is clear that $y \neq 0$. Now $ v_t<_r v_s $ and $\varphi_{\max}^{<_r}(v_t,v_s)=\varphi(v_tP_2v_s)$, for some $ v_tP_2v_s \in \mathcal{P}^{\max}(v_s,v_t) $. Then $\varphi_{\max}^{<_r}(v_s,v_t)=\overline{\varphi_{\max}^{<_r}(v_t,v_s)}=\overline{\varphi(v_tP_2v_s)}=\varphi(v_sP_2v_t)$. Since  $ v_tP_2v_s \in \mathcal{P}^{\max}(v_s,v_t) $, so $ \varphi(v_sP_2v_t)=x-iy_1 \in \mathbb{T} $, where either $ y_1= y$ or $ y_1=-y $. Also, $ \varphi_{\max}^{<}(v_s,v_t)\ne \varphi_{\max}^{<_r}(v_s,v_t)$, so $ \varphi(v_sP_1v_t)\ne \varphi(v_sP_2v_t) $. Thus $ \varphi(v_sP_2v_t)=x-iy $ and $ y>0 $. Hence $ (i) $ holds.
	Similarly, $ D^{\min}_{<}(\Phi)\ne D^{\min}_{<_r}(\Phi)$ implies $ (ii)$.
	
	Conversely, suppose statement $ (i) $ holds. Then there exist two shortest $ \overrightarrow{v_sv_t} $-paths $ v_sP_1v_t $ and $ v_sP_2v_t $ in $ \mathcal{P}^{\max}(v_s,v_t) $ with different gains. If $ \varphi(v_sP_1v_t )=x+iy \in \mathbb{T}$, then  $ \varphi(v_sP_2v_t)=x-iy$, $ y\neq0 $. Without loss of generality, assume that $y>0$. If $ v_s<v_t $, then $ \varphi_{\max}^{<}(v_s,v_t)=x+iy $ and $ \varphi_{\max}^{<_r}(v_s,v_t)=\overline{\varphi_{\max}^{<_r}(v_t,v_s)}=\overline{x+iy}=x-iy$. Thus $ \varphi_{\max}^{<}(v_s,v_t)\ne\varphi_{\max}^{<_r}(v_s,v_t)$. Therefore $D^{\max}_{<}(\Phi)\ne D^{\max}_{<_r}(\Phi) $ and hence $ \Phi $ is not order independent. Similarly if the statement $ (ii) $ holds, then $ \Phi $ is not order independent.
\end{proof}

\begin{prop}\label{prop.3.1}
	Let $ \Phi=(G, \varphi) $ be a $ \mathbb{T} $-gain graph and $ `< `$ be the standard vertex order. Then $ D_{<}^{\max}(\Phi)= D_{<}^{\min}(\Phi)$ if and only if $ D^{\max}(\Phi)$ and $ D^{\min}(\Phi)$ are well defined and $D^{\max}(\Phi)= D^{\min}(\Phi)$.
\end{prop}
\begin{proof}
	Suppose $ D_{<}^{\max}(\Phi)= D_{<}^{\min}(\Phi)$. Let $ s,t \in V(G)$  such that $ s<t $. Then $ d^{<}_{\max}(s,t)=d^{<}_{\min}(s,t) $. Therefore $ \varphi^{\max}_{<}(s,t)=\varphi^{\min}_{<}(s,t) $. Thus all the shortest paths from $ s $ to $ t $ have the same gain. Therefore, $\varphi^{\max}_{<}(s,t)=\varphi^{\max}_{<_r}(s,t)$ and $\varphi^{\min}_{<}(s,t)=\varphi^{\min}_{<_r}(s,t)$. Thus $ D^{\max}(\Phi) $ and $ D^{\min}(\Phi)  $ are well defined. Since $ d^{<}_{\max}(s,t)=d^{<}_{\min}(s,t) $, so $ D^{\max}(\Phi)=D^{\min}(\Phi)$.
	The converse is easy to verify.
\end{proof}

\begin{theorem}\label{cor3.1}
	Let $ \Phi=(G, \varphi) $ be a $ \mathbb{T} $-gain graph and $ `< `$ be the standard vertex order. Let $  <_a $ be any vertex order on $V(G)$. Then $ D^{\max}_{<}(\Phi)=D^{\min}_{<}(\Phi) $ if and only if $ D^{\max}_{<_a}(\Phi)=D^{\min}_{<_a}(\Phi) $.
\end{theorem}
\begin{proof}
	Let $ < $ be the standard vertex order, and $ D^{\max}_{<}(\Phi)=D^{\min}_{<}(\Phi) $. Let $ v_i, v_j \in V(G)$. Then $ d^{<}_{\max}(v_i,v_j)=d^{<}_{\min}(v_i,v_j) $ and  $ \varphi^{<}_{\max}(v_i,v_j)=\varphi^{<}_{\min}(v_i,v_j)$. Thus all the shortest paths from $ v_i $  to $ v_j $ have the same gain. Therefore, for any arbitrary vertex ordering $ <_a $, we have $ \varphi^{<_a}_{\max}(v_i,v_j)=\varphi^{<_a}_{\min}(v_i,v_j)$. Hence $ d^{<_a}_{\max}(v_i,v_j)=d^{<_a}_{\min}(v_i,v_j) $. Since $ v_i $ and $ v_j $ are arbitrary, so $ D^{\max}_{<_a}(\Phi)=D^{\min}_{<_a}(\Phi) $.
	
Proof of the converse is similar to that of the previous part.
\end{proof}
\begin{definition}[Distance compatible]{\rm
		A $ \mathbb{T} $-gain graph $ \Phi=(G, \varphi) $ is called \emph{gain distance compatible (simply, distance compatible)} if $ D^{\max}_{<}(\Phi)=D^{\min}_{<}(\Phi) $, where $  < $ is the standard order. In this case, we define $ D(\Phi)=D^{\max}_{<}(\Phi)=D^{\min}_{<}(\Phi)$. }
\end{definition}
The proof of the following theorem is easy to verify.
\begin{theorem}\label{dist-comp-well}
	For a  $ \mathbb{T} $-gain graph $ \Phi =(G, \varphi)$, the following are equivalent:
	\begin{enumerate}
		\item[(1)]  $ \Phi $ is distance compatible.
		\item [(2)] $ D(\Phi)=D^{\max}_{<}(\Phi)= D^{\min}_{<}(\Phi)=D^{\max}(\Phi)=D^{\min}(\Phi)$.
		\item [(3)]  $ D(\Phi) $ is well defined.
	\end{enumerate}
	
\end{theorem}
\begin{proof}
	$ (1) \implies (2):$ Let $ \Phi $ be distance compatible. Then, by the definition, $ D^{\max}_{<}(\Phi)=D^{\min}_{<}(\Phi) $ for standard vertex order $ < $. Also $ D(\Phi)= D^{\max}_{<}(\Phi)=D^{\min}_{<}(\Phi) $. Now by Proposition \ref{prop.3.1}, $ D^{\max}(\Phi)=D^{\min}(\Phi) $. Hence  $ D^{\max}(\Phi)=D^{\min}_{<}(\Phi)$ and $ D^{\min}(\Phi)=D^{\min}_{<}(\Phi)$. \\
	$ (2)\implies(3):$ By the definition of distance-compatible $ \mathbb{T} $-gain graph, $ D(\Phi) $ exists.\\
	$(3)\implies(1):$ If $ D(\Phi) $ is well defined, then $ D(\Phi)=D^{\max}(\Phi)=D^{\min}(\Phi) $. Hence $ \Phi $ is distance-compatible.
	
\end{proof}

\begin{prop}\label{prop3.2}
	Let $ \Phi=(G, \varphi) $ be any distance compatible $ \mathbb{T} $-gain graph. If $ \Phi \sim \Psi $, then $ \Psi $ is distance compatible and $ \spec(D(\Phi))=\spec(D(\Psi))$.
\end{prop}
\begin{proof}
	Let $ \Phi=(G, \varphi) $ be a $ \mathbb{T} $-gain graph with the standard  vertex order $  < $. Let $ s,t\in V(G) $. Since $ \Phi $ is distance compatible,  all oriented shortest paths $ sPt$ from $ s $  to $ t $ have the same gain. As $ \Phi \sim \Psi $, so   there exists a switching function $ \zeta $ such that $ \psi(sPt)=\zeta(s)^{-1}\varphi(sPt) \zeta(t)$, for any shortest path $ sPt$. For any shortest path $ sPt $, $ \varphi(sPt) $ is unique, so   $ \psi(sPt) $ is unique. Thus $ \psi_{\max}^{<}(s,t)=\psi_{\min}^{<}(s,t) $  and hence $ D^{\max}_{<}(\Psi)=D^{\min}_{<}(\Psi) $. That is, $ \Psi $ is distance compatible and $ D(\Psi) $ is well defined. Let $ d_{\psi}(s,t) $ and $ d_{\varphi}(s,t) $ be the unique gain distance from $ s $ to $ t $ in $ \Psi $ and $ \Phi $, respectively. Then $ d_{\varphi}(s,t) =\zeta(s)^{-1} d_{\psi}(s,t)\zeta(t)$. Thus $ D(\Phi) $ and $ D(\Psi) $ are similar and hence $ \spec(D(\Phi))=\spec(D(\Psi))$.
\end{proof}

Converse of the above statement holds for balanced $ \mathbb{T} $ gain graph, see Corollary \ref{cor5.1}.

\begin{prop}\label{prop3.3}
	Let $ \Phi=(G, \varphi) $ be a $ \mathbb{T} $-gain graph. Then
	\begin{enumerate}
		\item[(1)] If $ \Phi $ is a signed graph, then $ \Phi $ is order-independent.
		\item[(2)] If $ \Phi $ is balanced or anti-balanced, then $ \Phi $ is order-independent.
		\item[(3)] If $ \Phi $ is distance compatible, then $ \Phi $ is order-independent.
		\item[(4)] If $ \Phi $ is geodetic, then $ \Phi $ is order-independent.
	\end{enumerate}
\end{prop}
However, converse of  the above statements need not be true in general, see Example \ref{ex1}. The following result  is an extension of the Harary's path criterion for  $ \mathbb{T} $-gain graphs.
\begin{lemma}\label{lm3.1}
	Let $ \Phi=(G, \varphi) $ be any $ \mathbb{T} $-gain graph. Then $ \Phi $ is balanced if and only if every directed $ (s,t) $-path have the same gain in $ \Phi $, for any two vertices $ s, t $.	
\end{lemma}
\begin{proof}
	Let $ \Phi=(G, \varphi) $ be balanced. Suppose that the two oriented paths $ sP_1t$ and $ sP_2t $ have different gains. That is $ \varphi(sP_1t) \ne \varphi(sP_2t)$. Then $ \varphi(sP_1t) \varphi(tP_2s)=\sum\limits_{j=1}^{k}\varphi(\overrightarrow{C_j})\ne 1$, where $ C_1, C_2, \dots, C_k $ are the cycles formed by these two paths. Therefore, there exist at least one cycle, say $C_j$ such that $ \varphi(\overrightarrow{C_j}) \ne 1$ . Thus $ \Phi $ is not balanced, a contradiction. Converse is easy to verify.
\end{proof}

Let $ \Phi=(G, \varphi) $ be any $ \mathbb{T} $-gain graph. Then $ \Phi $ is either order-independent or order-dependent. If $ \Phi $ is ordered-independent, then $ \Phi $  may or may not be balanced, anti-balanced, geodetic. If $ \Phi $ is ordered-dependent, then, by Proposition \ref{prop3.3}, $ \Phi $ is neither balanced nor anti-balanced nor geodetic. Therefore, any $ \mathbb{T} $-gain graph $ \Phi=(G, \varphi) $ belongs to one of the following classes:
\begin{enumerate}
	\item [(A)] $ \Phi $ is balanced or anti-balanced or geodetic and $ D^{\max}(\Phi)=D^{\min}(\Phi) $.
	\item [(B)] $ \Phi $ is neither  balanced nor anti-balanced nor geodetic and $ D^{\max}(\Phi)=D^{\min}(\Phi) $.
	\item [(C)]$ \Phi $ is neither balanced nor anti-balanced nor geodetic and $ D^{\max}(\Phi)\ne D^{\min}(\Phi) $.
	\item [(D)] $ \Phi $ is neither balanced nor anti-balanced nor geodetic and at least one of $ D^{\max}(\Phi)$ and $D^{\min}(\Phi) $ are not well defined.
\end{enumerate}

Next we give some examples. Examples of  Type $ (A) $ can be constructed easily. Example \ref{ex3.1} is of Type $ (D) $.  Examples of Type $(B)$ and Type $(C)$ are given below.
\begin{figure} [!htb]
	\begin{center}
		\includegraphics[scale= 0.55]{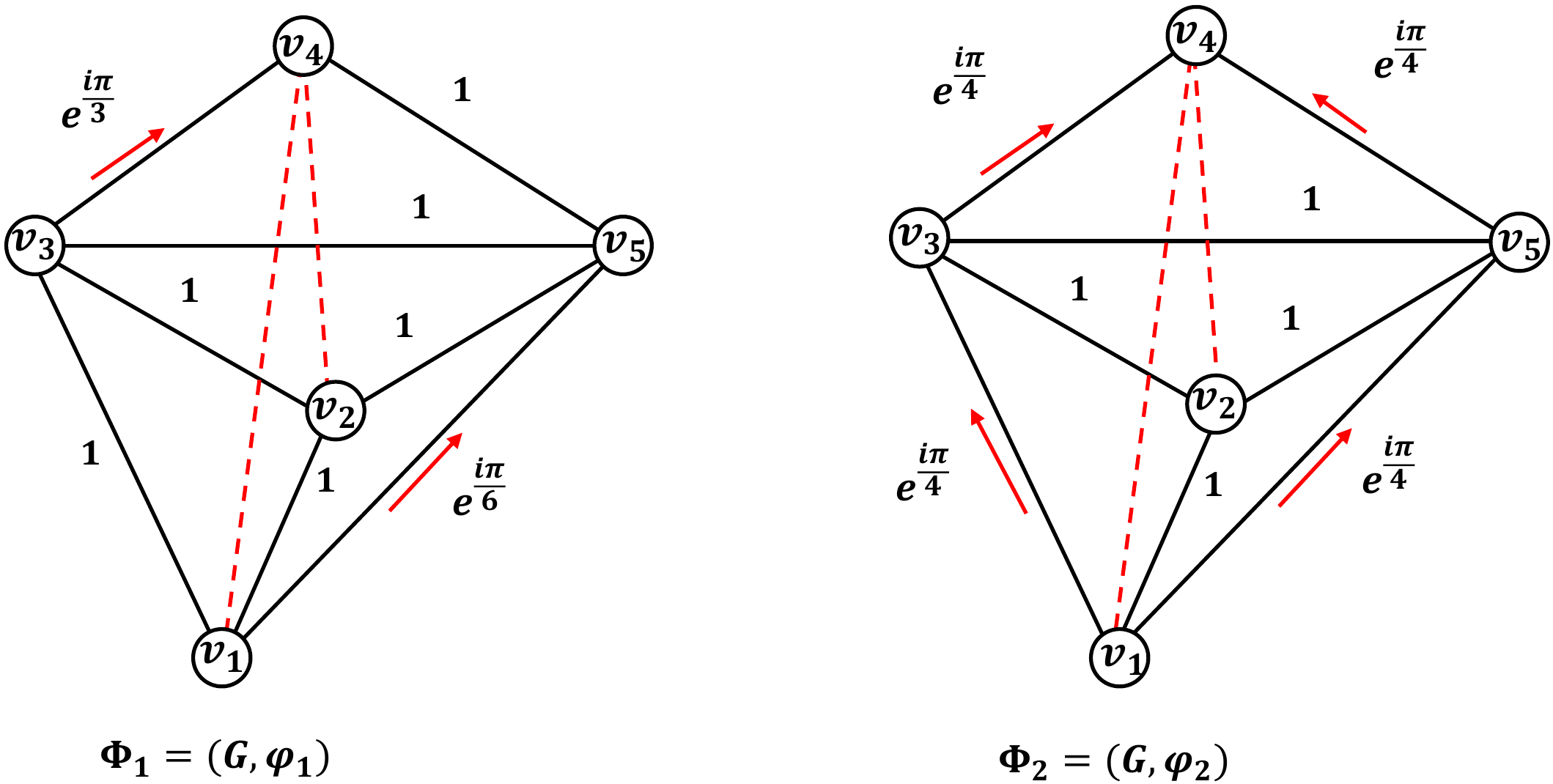}
		\caption{ $ \mathbb{T} $-gain graphs $ \Phi_1 $ and $ \Phi_2 $} \label{fig1.1}
	\end{center}
\end{figure}

\begin{example}\label{ex1}{\rm
	Let us consider the $ \mathbb{T} $-gain graph $ \Phi_1=(G, \varphi_1) $ (in Figure \ref{fig1.1}) with the standard vertex order. The graph $ \Phi_1 $ is neither balanced nor anti-balanced nor geodetic. Also $ \Phi $ is order-independent. Now
	\begin{equation*}
		D^{\max}(\Phi_1)=\left[\begin{array}{ccccc}
			0 & 1 & 1 & 2e^{\frac{i\pi}{6}} & e^{\frac{i\pi}{6}}\\
			1 & 0 & 1 & 2 & 1\\
			1 & 1 & 0 & e^{\frac{i\pi}{3}} & 1\\
			2e^{-\frac{i\pi}{6}} & 2 & e^{-\frac{i\pi}{3}} & 0 & 1\\
			e^{-\frac{i\pi}{6}} & 1 & 1 & 1 & 0
		\end{array} \right],
		D^{\min}(\Phi_1)=\left[ \begin{array}{ccccc}
			0 & 1 & 1 & 2e^{\frac{i\pi}{3}} & e^{\frac{i\pi}{6}}\\
			1 & 0 & 1 & 2e^{\frac{i\pi}{3}} & 1\\
			1 & 1 & 0 & e^{\frac{i\pi}{3}} & 1\\
			2e^{-\frac{i\pi}{3}} & 2e^{-\frac{i\pi}{3}} & e^{-\frac{i\pi}{3}} & 0 & 1\\
			e^{-\frac{i\pi}{6}} & 1 & 1 & 1 & 0
		\end{array} \right].
	\end{equation*}
	Thus $ D^{\max}(\Phi_1)\ne D^{\min}(\Phi_1) $. Therefore, by Theorem \ref{dist-comp-well}, $ \Phi_1$ is distance incompatible.}
\end{example}
\begin{example}\label{eq3.3}{\rm
	The $ \mathbb{T} $-gain graph $ \Phi_2 $ (in Figure \ref{fig1.1}) with the standard vertex ordering is neither balanced nor anti-balanced nor geodetic. Here $ \Phi_2$ is order-independent. However, it is distance compatible and
	\begin{equation*}
		D^{\max}(\Phi_2)=D^{\min}(\Phi_2)=	\left[\begin{array}{ccccc}
			0 & 1 & e^{\frac{i\pi}{4}} & 2e^{\frac{i\pi}{2}} & e^{\frac{i\pi}{4}}\\
			1 & 0 & 1 & 2e^{\frac{i\pi}{4}} & 1\\
			e^{-\frac{i\pi}{4}} & 1 & 0 & e^{\frac{i\pi}{4}} & 1\\
			2e^{-\frac{i\pi}{2}} & 2e^{-\frac{i\pi}{4}} & e^{-\frac{i\pi}{4}} & 0 & e^{-\frac{i\pi}{4}}\\
			e^{-\frac{i\pi}{4}} & 1 & 1 & e^{\frac{i\pi}{4}} & 0
		\end{array} \right].
	\end{equation*}}
\end{example}

\section{Positively weighted $ \mathbb{T} $-gain graph}\label{Positively_weighted}
In this section, we introduce the notion of a positively weighted $ \mathbb{T} $-gain graph. The adjacency matrices of positively weighted $ \mathbb{T} $-gain graphs generalize the following notions: $ \mathbb{T} $-gain adjacency matrices, Hermitian adjacency matrices of mixed graphs, adjacency matrices of signed graphs, adjacency matrices of undirected graphs.
\begin{definition}[Positively weighted $ \mathbb{T} $-gain graph]{\rm
	Let $ \Phi=(G, \varphi) $ be any $ \mathbb{T} $-gain graph with $ E(G) $ be the undirected edge set of $ G $. Let $ w:E(G)\rightarrow \mathbb{R}^{+} $ be a  weight function on the edges of $ G $. The positively weighted $ \mathbb{T} $-gain graph associated with $ \Phi $ and $ w $ is the graph $G $ together with the weighted gain function $ \varphi_w $ defined as follows:
	\begin{center}
		$ \varphi_w(\overrightarrow{e}_{i,j})=\varphi(\overrightarrow{e}_{i,j})w(e_{i,j}). $
	\end{center}
	A positively weighted $ \mathbb{T} $-gain graph on $(G, \varphi)$ is denoted by $ (G, \varphi, w) $(or simply $ \Phi_{w} $).}
	
\end{definition}
The adjacency matrix associated with $ \Phi_{w} $,  denoted by $ A(\Phi_{w}) $,  is  an $ n\times n $ Hermitian matrix whose $ (i,j)th $ entry is $\varphi_w(\overrightarrow{e}_{i,j})$ if $e_{i,j} \in E(G)$, and zero otherwise. Since $ A(\Phi_w) $ is Hermitian, so all its eigenvalues are real.  The spectrum of $ A(\Phi_{w}) $ is the spectrum of $ \Phi_{w} $.

If $ \varphi=1 $, then the corresponding positively weighted $ \mathbb{T} $-gain graph is the weighted graph $ (G, w) $,  and is denoted by $ G_{w} $. The adjacency matrix of $ G_w $, denoted by $ A(G_w) $,  is an  $n\times n $ symmetric matrix with the $ (i,j) $th entry  $ w(e_{i,j}) $. Then $ A(\Phi_{w})=A(\Phi) \circ A(G_w) $, where $ '\circ' $ is the Hadamard product.

We establish an expression for the characteristic polynomial of $ A(\Phi_w) $, which is a generalization of the weighted Sachs formula. Let $ \Phi_{w}=(G, \varphi, w) $ be a positively weighted $ \mathbb{T} $-gain graph. The weight of a cycle $ C $ in $ G_w $ is defined as $ w(C)=\prod\limits_{e\in E(C)}w(e) $. Now $ \overrightarrow{C} $ is an oriented cycle. Then the weighted $ \mathbb{T} $-gain of $ \overrightarrow{C} $ is $ \varphi_w(\overrightarrow{C})=\prod\limits_{\overrightarrow{e}\in \overrightarrow{E(C)}}\varphi_w(\overrightarrow{e})=\prod\limits_{\overrightarrow{e}\in \overrightarrow{E(C)}}\varphi(\overrightarrow{e})w(e)=w(C)\varphi(\overrightarrow{C})$. An elementary subgraph $ H $ of $ G $ is a subgraph of $ G $ such that each component of $H$ is either a cycle or an edge of $ G $.
For an elementary subgraph $ H$, $ H_{e} $ denotes the set of isolated edges in $ H $. The collection of all elementary subgraphs with $ k $ vertices is denoted by $ \mathcal{H}_{k} $.   Next, we state the weighted gain Sachs formula. As the proof is similar to that of the weighted case, so we skip it.

\begin{theorem}[{Weighted gain Sachs formula}]\label{Th3.1}
	Let $ \Phi_w=(G, \varphi, w) $ be a positively weighted $ \mathbb{T} $-gain graph with characteristic polynomial $ \chi(\Phi_{w};x)=x^{n}+a_1x^{n-1}+\dots+a_n $. Then
	\begin{equation}
		a_i=\sum\limits_{H\in \mathcal{H}_i}(-1)^{p(H)}2^{c(H)}w(H_e) w(H)\prod\limits_{C\in H}\Rea(\varphi(C)),
	\end{equation}
	where $ c(H), p(H)$ and $ C $ denote the number of cycles, the number of components and cycle in $ H $ , respectively.
\end{theorem}

If we choose $ \mathbb{T} $-gain graph $ \Phi $ to be the underlying graph $ G $, then above formula become the known Weighted Sachs formula.
\begin{corollary}[Weighted Sachs formula, \cite{Sign_distance}]
	Let $ (G, w) $ be a weighted graph with characteristic polynomial $ \chi(G, w;x)=x^{n}+a_1x^{n-1}+\dots+a_n $. Then the coefficients can be expressed as
	\begin{equation}
		a_i=\sum\limits_{H\in \mathcal{H}_i}(-1)^{p(H)}2^{c(H)}w(H_e) w(H),
	\end{equation}
	where $ c(H), p(H) $ denote the number of cycles and the number of components in $ H $, respectively.	
\end{corollary}

Now we are ready to state two interesting results which generalize the corresponding known result for $ \mathbb{T} $-gain graph and signed graph.
\begin{theorem}\label{th4.2}
	Let $ \Phi_{w}=(G, \varphi, w) $ be a positively weighted $ \mathbb{T} $-gain graph. Then $ \Phi_{w} $ and $ G_w$ are cospectral if and only if $ \Phi $ is balanced.
\end{theorem}
\begin{proof}
	If $ \Phi=(G, \varphi) $ is balanced, then there exists a diagonal unitary matrix $ U $ such that $ A(\Phi)=UA(G)U^{*} $. Now, by Proposition \ref{prop3.1},
	\begin{align*}
		A(\Phi_w)=A(\Phi)\circ A(G_w)&=UA(G)U^{*} \circ A(G_w)=U(A(G) \circ A(G_w))U^{*}=UA(G_w)U^{*}.
	\end{align*} Thus $ \Phi_{w} $ and $ G_w$ are cospectral
	
	Conversely, suppose $ \Phi_w $ and $ G_w $ are cospectral. Let $ \chi(\Phi_{w};x)=\sum\limits_{i=0}^{n}a_ix^{n-i} $ and $ \chi(G,w;x)=\sum\limits_{i=0}^{n}b_ix^{n-i} $ be the characteristic polynomials of $\Phi_w$ and $G_w$, respectively. Suppose that $ \Phi $ is not balanced. Then there exists a cycle of  smallest length, say $ k $, which is not balanced. Let $\mathcal{C}_{k}$ be the collection of all unbalanced $ k $ cycles. Then, by Theorem \ref{Th3.1}, $$b_k-a_k=2\sum\limits_{C\in \mathcal{C}_k}w(C).\{ 1-\Rea(\varphi(C))\}>0, $$
	a contradiction. Thus $ \Phi $ is balanced.
\end{proof}


The well known Acharya's spectral criterion for the balance  of signed graphs follows from  Theorem \ref{th4.2}.
\begin{corollary} [{\cite[Corollary 1.1]{Acharya}}]
	Let $ \Psi=(G, \psi) $ be a signed graph. Then spectra of $ \Psi $ and $ G $ coincide if and only if $ \Psi $ is balanced.
\end{corollary}
\begin{proof}
	By taking $ \varphi=\pm 1 $ and $ w=1 $ in Theorem  \ref{th4.2}, we get the result.
\end{proof}
Another consequence is the following recent result about the signed graph.
\begin{corollary}[{\cite[Theorem 2.4]{Sign_distance}}]
	Let $ \Psi=(G, \psi)$  be a signed graph and $ w $ be a positively weighted function, where $ \psi=\pm 1 $. Then $ \Psi_w $ and $ G_w $ are cospectral if and only if $ \Psi $ is balanced.
\end{corollary}

Next, we prove one of the main results of this article.

\begin{theorem}\label{Th3.4}
	Let $ \Phi_w=(G, \varphi, w) $ be a connected positively weighted $\mathbb{T}$-gain graph. Then the spectral radius of  $ \Phi_w $ and $ G_w$ are equal if and only if either $ \Phi $ or $ -\Phi $ is balanced.
\end{theorem}
\begin{proof}
	Suppose either $ \Phi $ or $ -\Phi $ is balanced. Then, it is easy to see that, the spectral radius $ \Phi_{w} $ and $ G_w $ are equal.
	Conversely, suppose $ \rho(\Phi_{w})=\rho(G_w) $. Let $ \mu_n\leq \mu_{n-1}\leq \dots \leq \mu_1 $ be the eigenvalues of $ \Phi_{w} $. Then either $ \rho(\Phi_{w})=\mu_1 $ or $ \rho(\Phi_{w})=-\mu_n $.
	
	\noindent{\bf Case 1:} If $ \rho(\Phi_{w})=\mu_1 $, then, by Theorem \ref{th2.2}, there exists a diagonal unitary matrix $ D $ such that $ A(\Phi_{w})=DA(G_w)D^{*} $. Now $ A(\Phi)\circ A(G_w)=D(A(G) \circ A(G_w))D^{*} $. Then, by Proposition \ref{prop3.1}, $ A(\Phi)\circ A(G_w)=(DA(G)D^{*} )\circ A(G_w)$. Define  $ B=(b_{ij}) $ as follows:   $ b_{ij} $ is the inverse of the nonzero $ (i,j)th $-entry of $ A(G_w) $, otherwise zero. Then $( A(\Phi)\circ A(G_w))\circ B=((DA(G)D^{*} )\circ A(G_w))\circ B$. Thus, by Proposition \ref{prop3.1}, we have $A(\Phi)=DA(G)D^{*}$. Thus $ \Phi $ is balanced.
	
	\noindent{\bf Case 2:} If $ \rho(\Phi_w)=-\mu_n$, then $ \mu_n=e^{i\pi} \rho(\Phi_w)$. By Theorem \ref{th2.2}, there exists a diagonal unitary matrix $ D $, such that $ A(\Phi_{w})=e^{i\pi}DA(G_w)D^{*} $. That is, $-A(\Phi_{w})=DA(G_w)D^{*} $. Then $ A(-\Phi)\circ A(G_w)=D(A(G) \circ A(G_w))D^{*} $. By Proposition \ref{prop3.1}, we have $A(-\Phi)=DA(G)D^{*}$. Thus $ -\Phi $ is balanced.
\end{proof}
Now we present the following consequences of the above results.
\begin{corollary}\label{Cor3.2}
	Let $ \Phi_w=(G, \varphi, w) $ be a connected positively weighted $\mathbb{T}$-gain graph. Then the largest eigenvalue of  $ \Phi_w $ and $ G_w $ are equal if and only if $ \Phi $ is balanced.	
\end{corollary}

\begin{corollary}\label{lm3.2}
	Let $ \Phi=(G, \varphi) $ be a connected $ \mathbb{T} $-gain graph. Then the largest eigenvalue of $ \Phi $ and $ G $ are equal if and only if $ \Phi $ is balanced.
\end{corollary}
\begin{proof}
	The proof follows from Corollary \ref{Cor3.2} by assuming $ w=1 $.
\end{proof}

Also Theorem \ref{Th3.4} unifies the following recent results.
\begin{corollary}[{\cite[Corollary 2.7]{Sign_distance}}]
	Let $ \Psi_w=(G, \psi,w) $ be a connected positively weighted signed graph. Then $ \Psi$ is balanced if and only if the largest eigenvalue of $ \Psi_w $ and $ G_w $ coincide.
\end{corollary}
\begin{proof}
	By taking $ \varphi=\pm1 $,	the result follows from Corollary \ref{Cor3.2}.
\end{proof}

\begin{corollary}[{\cite[Theorem 4.4]{Our-paper-1}}]
	Let $ \Phi=(G, \varphi) $ be a connected $ \mathbb{T} $-gain graph. Then spectral radius of $ \Phi $ and $ G $ coincide if and only if either $ \Phi $ is balanced or $ -\Phi $ is balanced.
\end{corollary}
\begin{proof}
	Take $ w=1 $ in Theorem \ref{Th3.4}.
\end{proof}

\begin{corollary}[{(Stani\'c's spectral criterion \cite[Lemma 2.1]{Stanic})}]
	Let $\Psi=(G, \psi) $ be a connected signed graph. Then the largest eigenvalue of $ \Psi $ and $ G $ coincide if and only if $ \Psi $ is balanced.
\end{corollary}
\begin{proof}
	Result follows from Corollary \ref{Cor3.2} by choosing $ \varphi=\pm 1 $ and $ w=1 $.
\end{proof}
\section{Characterizations of balanced $ \mathbb{T} $-gain graphs in terms of gain distance matrices}\label{Char_of_balanced}

In this section, we establish two characterizations for balanced $ \mathbb{T} $-gain graphs using the gain distance matrices. Let us define two complete $ \mathbb{T} $-gain graphs which are obtained from gain distance matrices $ D^{\max}_{<}(\Phi) $ and $ D^{\min}_{<}(\Phi) $.

\begin{definition}{\rm
	Let $ \Phi=(G, \varphi) $ be a $ \mathbb{T} $-gain graph and $`<`$ be an order on $V(G)$.  The complete $ \mathbb{T} $-gain graph with respect to $D^{\max}_{<}(\Phi) $, denoted by $ K^{D^{\max}_{<}}(\Phi) $,  is defined as follows:  keep  the edges of  $ \Phi $ unchanged, and join  non adjacent    vertices $v_i$ and $v_j$ with gain $ \varphi(\overrightarrow{e}_{i,j})= \varphi^{<}_{\max}(v_i,v_j)$ for all $v_i,v_j$. Similarly
 $ K^{D^{\min}_{<}}(\Phi) $ is defined  using  $D^{\min}_{<}(\Phi) $.}
	\end{definition}

For a $ \mathbb{T} $-gain graph $ \Phi=(G, \varphi) $ with order $<$, if $ D^{\max}_{<}(\Phi)=D^{\min}_{<}(\Phi) $, then the associated complete $ \mathbb{T} $-gain graphs $ K^{D^{\max}}(\Phi) $ and $K^{D^{\max}}(\Phi)  $ are the same,  and it is denoted by $ K^{D}(\Phi) $. Then the proof of the following theorem is easy to verify.

\begin{theorem}
	For a $ \mathbb{T} $-gain graph $ \Phi=(G, \varphi) $, the following are equivalent:
	\begin{enumerate}
		\item [(1)]  $ K^{D}(\Phi) $ is well defined.
		\item [(2)] $ K^{D^{\max}}(\Phi)=K^{D^{\max}}(\Phi) = K^{D^{\max}_{<}}(\Phi)=K^{D^{\min}_{<}}(\Phi)=K^{D}(\Phi) $.
		\item [(3)] $ D^{\max}_{<}(\Phi)=D^{\min}_{<}(\Phi)$, for some ordering $ < $.
	\end{enumerate}
\end{theorem}

\begin{theorem}\label{Th5.1}
	Let $ \Phi=(G, \varphi) $ be a $ \mathbb{T} $-gain graph with vertex order $<$. Then the following statements are equivalent.
	\begin{enumerate}
		\item [(i)] $ \Phi $ is balanced.
		\item [(ii)] $ K^{D^{\max}}(\Phi) $ is balanced.
		\item [(iii)] $ K^{D^{\min}}(\Phi) $ is balanced.
		\item [(iv)] $ D^{\max}(\Phi) =D^{\min}(\Phi)$ and associated complete $ \mathbb{T} $-gain graph $ K^{D}(\Phi) $ is balanced.
	\end{enumerate}
\end{theorem}
\begin{proof}
	$ (i) \implies (iv) $	Let $ V(G)=\{v_1, v_2, \dots, v_n\} $. Suppose $ \Phi $ is balanced. Let $ v_i, v_j\in V(G) $. Then, by Lemma \ref{lm3.1}, all shortest oriented paths $ v_iPv_j$ have the same gain. Thus $ \varphi_{\max}^{<}(v_i,v_j)=\varphi_{\min}^{<}(v_i,v_j) $. Therefore, $ D^{\max}_{<}(\Phi)=D^{\min}_{<}(\Phi) $. By Proposition \ref{prop.3.1} and Corollary \ref{cor3.1}, $ D^{\max}(\Phi)=D^{\min}(\Phi) $. Hence  $ K^{D}(\Phi) $ is well defined.\\
	{\bf Claim:} $ K^{D}(\Phi) $ is balanced.\\
	Let $ v_i\nsim v_j $ in $ G $ and $ e_{i,j} $ be the edge joining $v_i$ and $v_j$ in $ K^{D}(\Phi)$. For every oriented path $ v_iPv_j$ in $ \Phi $ have the same gain.  So every cycle passing through the edge $ e_{i,j} $ has gain $ 1 $. Let $ T $ be a normal spanning tree of $ G $. Suppose $ v_i\nsim v_j $ in $ G $. In $ T $, joining the edge $ e_{i,j} $ creates a fundamental cycle of $K^{D}(\Phi) $, say $ C_T $. Now by previous observation, $ \varphi(C_T)=1$. Thus all the fundamental cycles in $ K^{D}(\Phi)  $ are neutral, and hence, by Lemma \ref{lm2.1}, $ K^{D}(\Phi)$ is balanced.

If $ K^{D}(\Phi) $ is balanced, then,  as $ \Phi $ is a subgraph of $ K^{D}(\Phi) $, so $ \Phi $ is balanced. Therefore, $ (iv) \implies (i), (iii) \implies (i)$ and $ (ii) \implies (i) $ follow.

	The proofs  $ (iv) \implies (iii) $ and $ (iv) \implies (ii) $ are easy to see.
\end{proof}
\begin{theorem}\label{th5.2}
Let $ \Phi=(G, \varphi) $ be a $ \mathbb{T} $-gain graph with vertex order $<$.Then the following statements  are equivalent:
	\begin{enumerate}
		\item [(i)] $ \Phi $ is balanced.
		\item [(ii)] $ D^{\max}(\Phi) $ is cospectral with $ D(G) $.
		\item [(iii)]$ D^{\min}(\Phi) $ is cospectral with $ D(G) $.
		\item[(iv)] The largest eigenvalue of $D^{\max}(\Phi)   $  and $ D(G) $ are  equal.
		\item [(v)]The largest eigenvalue of $D^{\min}(\Phi)   $  and $ D(G) $ are  equal.
	\end{enumerate}
\end{theorem}

\begin{proof}

	Let $ V(G)=\{ v_1, v_2, \dots, v_n\} $, and $ \Phi $ be balanced.  Then by Theorem \ref{Th5.1}, $ K^{D^{\max}}(\Phi)$ is balanced. Note that $K^{D^{\max}}(\Phi)=(K_n, \psi)  $ with $ \psi(\overrightarrow{e}_{i,j})=\varphi_{\max}(v_i,v_j)=\varphi(v_iPv_j)$, where $v_iPv_j$  is a shortest path in $ \Phi $. Consider $ {D^{\max}}(\Phi)$ as the adjacency matrix of a positively weighted $ \mathbb{T} $-gain graph $ (K_n, \psi, w) $  with weight function $w: E(K_n) \rightarrow \mathbb{R}^{+}$ is defined as $ w(e_{i,j})=d(v_i,v_j) $, where $ d(v_i,v_j) $ is the distance between $ v_i$ and $ v_j $ in $ G $. Then the adjacency matrix of $ (K_n,  w)$ is same as  $D(G) $. By  Theorem \ref{th4.2}  and Theorem \ref{Th5.1}, $ \Phi $ is balanced if and only if $ D^{\max}(\Phi) $ is cospectral with $ D(G) $. Thus $(i) \Leftrightarrow (ii)$.
	
	Now,  by Corollary \ref{Cor3.2}, $ \Phi $ is balanced if and only if the largest eigenvalue of $D^{\max}(\Phi)   $  and $ D(G) $ coincide. This proves 	$ (i) \Leftrightarrow (iv) $.
	
	The proofs of  $ (i) \Leftrightarrow (iii) $ and $ (i) \Leftrightarrow (v) $ are similar .		 
\end{proof}
Both of the above characterizations extend the corresponding known characterizations \cite[Theorem 3.1]{Sign_distance} and \cite[Theorem 3.5]{Sign_distance} for signed graph.
\begin{corollary}\label{cor5.1}
	Let $ \Phi=(G, \varphi) $ be a $ \mathbb{T} $-gain graph. Then $ \Phi $ is balanced if and only if $ D(\Phi) $ exist and it is cospectral with $ D(G) $.
\end{corollary}
\section{Distance compatible gain graphs}\label{Dist_compt}
	
In this final section, we establish a couple of characterizations for distance compatible $ \mathbb{T} $-gain graphs. These results extend the corresponding known results for signed graph \cite{Sign_distance}.
\begin{theorem}\label{Th6.1}
	Let $ \Phi=(G, \varphi) $ be any bipartite $ \mathbb{T} $-gain graph. Then $ \Phi $ is distance compatible if and only if $ \Phi $ is balanced.
\end{theorem}
\begin{proof}
	If $ \Phi $ is balanced, by Theorem \ref{th5.2}, $ \Phi $ is distance compatible. Conversely, suppose $ \Phi $ is distance compatible. Then, by Proposition \ref{prop3.3}, $ \Phi $ is order-independent and $ D^{\max}(\Phi)=D^{\min}(\Phi) $.  Suppose that $ \Phi $ is unbalanced. Since $ \Phi $ is bipartite,  there exists an unbalanced  even cycle $ C $. Let $ v_i $ and $ v_j $ be two diametrical vertices of $ C $. Then $ C $ contains two disjoint paths $ v_iP_1v_j $ and $ v_iP_2v_j $ of  same length. Since $ \varphi(C)\ne 1 $ and $ \varphi(\overrightarrow{C})=\varphi(v_iP_1v_j)\varphi(v_jP_2v_i)\ne 1$, so $ \varphi(v_iP_1v_j)\ne \varphi(v_iP_2v_j) $.\\
	{\bf Claim:} $ v_iP_1v_j $ and $ v_iP_2v_j $ are shortest paths between $ v_i$  and $ v_j $.\\
	Suppose $ v_iP_1v_j $ and $ v_iP_2v_j $ are not shortest paths. Let  $ v_iPv_j $ be a shortest path. Then at least one of the even cycle formed by $ v_iP_1v_j  $, $ v_iPv_j $ and $ v_iP_2v_j  $, $ v_iPv_j $ is unbalanced, and has length strictly smaller than that of  $ C $, which is a contradiction. Since  $ \varphi(v_iP_1v_j)\ne \varphi(v_iP_2v_j) $,  for any ordered vertex set $ (V(G),<) $, $ \varphi_{\max}^{<}(v_i,v_j)\ne \varphi_{\min}^{<}(v_i,v_j) $, a contradiction. Thus $ \Phi $ is balanced.
\end{proof}
A \emph{cut vertex} in a graph $G$ is a vertex whose removal creates more components than the number of components of $G$. A \emph{block} of a graph $ G $ is a maximum connected subgraph of $ G $ that has no cut vertex.

\begin{theorem}\label{Th6.2}
	Let $ \Phi=(G, \varphi) $ be a $ \mathbb{T} $-gain graph. Then,  $ \Phi $ is distance compatible if and only if every block of  $\Phi$ is distance compatible.
\end{theorem}
\begin{proof}
	Let $ B_1,B_2, \dots, B_k $ be the blocks of $\Phi$. Suppose every block is distance compatible. Let $s,t\in V(G)$. If $ s $ and $ t $  are in the same block then they are distance compatible. Suppose $ s $ and $ t $ are in different blocks. Without loss of generality, suppose $ s $ is in $ B_1 $ and $ t $ is in $ B_2 $. Then any path $ sPt $ in $ G $ must passes through the cut vertices $ v_i $ and  $ v_j $ where $ v_i $ and $ v_j $ are in $ B_1 $ and $ B_2 $, respectively ($ v_i $ may be same as $ v_j $). Any shortest path $ sPt$  can be decompose into $ sPv_i\cup v_iPv_j\cup v_jPt $. Since $ B_1 $ is distance compatible, so any shortest path $ sPv_i $ has unique gain. As the vertices $ v_i $ and $ v_j $ are connected by a unique path, so $ \varphi(v_iPv_j) $ is unique. Proofs of the other cases are similar. Therefore, any shortest path from $ s $ to $ t $ has same gain. Thus $ s $ and $ t $ are distance compatible. Hence $ \Phi $ is distance compatible. Converse is easy to verify.
\end{proof}

Let $ \Phi=(G, \varphi) $ be $ \mathbb{T} $-gain graph with the standard order $<$ on the vertex set.  Let $ s,t \in V( G) $. If $ \varphi_{\max}^{<}(s,t)= \varphi_{\min}^{<}(s,t)$, then $ s $ and $ t $ are called distance compatible. Note that $ \varphi_{\max}^{<}(s,t)= \varphi_{\min}^{<}(s,t)$ if and only if $ \varphi_{\max}^{<_a}(s,t)= \varphi_{\min}^{<_a}(s,t)$, for any other vertex order $<_a$. Therefore, the vertices $ s $ and $ t $ are called \emph{distance-incompatible} if for some  order $<_a$,  $ \varphi_{\max}^{<_a}(s,t)\ne \varphi_{\min}^{<_a}(s,t) $ holds.
\begin{lemma}\label{lm7.1}
	Let $ \Phi=(G, \varphi) $ be a $ 2 $-connected non-geodetic $ \mathbb{T} $-gain graph. If $ s $ and $ t $ are two incompatible vertices of least distance in $ G $ then there exists at least two internally disjoint shortest paths between $ s $ and $ t $ which have different gains.
\end{lemma}
\begin{proof}
	Since $ s,t $ are distance-incompatible and $ G $ is non-geodetic, so there exist at least two shortest paths say $sP_1t $ and $sP_2t$ such that $ \varphi(sP_1t)\ne \varphi(sP_2t) $. If $sP_1t$ and $sP_2t $ are internally disjoint, then we are done. Suppose $sP_1t$ and $sP_2t $  are not internally disjoint. Let  $ v_1, v_2, \dots,v_p $ be the common internal vertices of the paths $sP_1t$ and $sP_2t $ . Let $ C_1, C_2, \dots, C_r $ be the only cycles formed by $sP_1t$ and $sP_2t $. Thus $ \varphi(sP_1t)\varphi(tP_2s)=\sum\limits_{i=1}^{r} \varphi(\overrightarrow{C_i}) \ne 1$. Then there exist a cycle $ C_j $ which is not balanced. Let $ C_j $ be formed by $ v_jP_1v_{j+1} $ and $ v_jP_2v_{j+1} $. Since $ sP_1t $ and $ sP_2t $ are shortest paths, so  $ v_jP_1v_{j+1} $ and $ v_jP_2v_{j+1} $ must be shortest paths in between $ v_j $ and $ v_{j+1} $ and of same lengths. Also  $ \varphi(\overrightarrow{C_j})=\varphi(v_jP_1v_{j+1}) \varphi(v_{j+1}P_2v_{j})\ne 1$. Thus $ \varphi(v_jP_1v_{j+1}) \ne \varphi(v_jP_2v_{j+1})$. Hence $ v_j $ and $ v_{j+1} $ are distance-incompatible, and   distance between them in $ G $ is smaller than the distance between the vertices $ s $ and $ t $ in $ G $,  a contradiction.
\end{proof}
\begin{theorem}\label{Th6.3}
	Let $ \Phi=(G, \varphi) $ be any $ 2 $-connected non-geodetic $ \mathbb{T} $-gain graph. Then $ \Phi $ is distance-incompatible if and only if there is an unbalanced even cycle such that there exist two diametrically opposite vertices $ s $ and $ t $ which have no other smaller length path.
\end{theorem}
\begin{proof}
	If $ \Phi $ is distance-incompatible, then there exist vertices $ s,t $  which are distance-incompatible and of least distance. Then by Lemma \ref{lm7.1}, there exists a pair of shortest disjoint paths in between $ s $ and $ t $ such that they have different gains. Let $ C_{2l} $ be the cycle formed by the two disjoin paths. Therefore, $ \varphi(\overrightarrow{C_{2l}})\ne 1$ and $ s,t $ do not have any other shorter length path.
	
	The converse is easy to verify.
\end{proof}
\section*{Acknowledgments}
The authors thank Prof Thomas Zaslavsky, Binghamton University, for his comments and suggestions, which improved the paper's presentation.	Aniruddha Samanta thanks University Grants Commission(UGC)  for the financial support in the form of the Senior Research Fellowship (Ref.No:  19/06/2016(i)EU-V; Roll No. 423206). M. Rajesh Kannan would like to thank the SERB, Department of Science and Technology, India, for financial support through the projects MATRICS (MTR/2018/000986) and Early Career Research Award (ECR/2017/000643).
	
	\bibliographystyle{amsplain}
	\bibliography{raj-ani-ref1}

\end{document}